\documentclass[smallextended]{svjour3_sta}       % nur fuer arXiv _sta
\usepackage{amsmath,amsthm1,amscd,amssymb}
\smartqed  % flush right qed marks, e.g. at end of proof
\usepackage{psfrag,contour}
\usepackage[mathscr]{eucal}
\usepackage{graphicx,url,color}
\newcommand\pfad{} % {fig_bil2/}
\newtheorem{thm}{Theorem}
\newtheorem{cor}[thm]{Corollary}
\newtheorem{lem}[thm]{Lemma}
\theoremstyle{definition}

\theoremstyle{remark}

\newcommand\Lemref[1]{Lemma~\ref{#1}}
\newcommand\Thmref[1]{Theorem~\ref{#1}}
\newcommand\Corref[1]{Corollary~\ref{#1}}

\newcommand\Figref[1]{Figure~\ref{#1}}
\numberwithin{equation}{section}
\newcommand\Vkt[1]{{\mathbf #1}}
\newcommand\dotVkt[1]{{\dot{\mathbf #1}}}
\newcommand\ol[1]{\overline{#1}}

\newcommand\wt[1]{\tilde{#1}}  % \widetilde
\newcommand\sz{\small} % 

\newlength{\lsqu}
\newcommand\Frac[2]{{\displaystyle\frac{#1}{#2}}}
\newcommand\smFrac[2]{\mbox{\small $\displaystyle\frac{#1}{#2}$}}
\newcommand\ssmFrac[2]{\mbox{\footnotesize $\displaystyle \frac{#1}{#2}$}}

\newcommand\RR{{\mathbb R}}

\newcommand\zwi{\mskip 9mu}
\newcommand\SN{\mskip 2mu\mathrm{sn}\mskip 2mu}
\newcommand\Sn{\mskip 2mu\mathrm{sn}}
\newcommand\CN{\mskip 2mu\mathrm{cn}\mskip 2mu}
\newcommand\Cn{\mskip 2mu\mathrm{cn}}
\newcommand\DN{\mskip 2mu\mathrm{dn}\mskip 2mu}
\newcommand\Dn{\mskip 2mu\mathrm{dn}}
\newcommand\const{\mathrm{const.}}
\definecolor{blau}{cmyk}{1.00,0.30,0.00,0.20} 
\newcommand\blue{\color{blau}}
\definecolor{rotcmyk}{cmyk}{0.00,1.00,1.00,0.00}
\newcommand\red{\color{rotcmyk}}
\definecolor{hgrau}{cmyk}{0.00,0.00,0.00,0.20}

\definecolor{dgrau}{cmyk}{0.00,0.00,0.00,0.85}
\newcommand\dgrau{\color{dgrau}}
\definecolor{gruen}{cmyk}{1.00,0.00,0.90,0.00} 
\def\green{\color{gruen}}
\definecolor{gelb}{cmyk}{0.00,0.00,1.00,0.00} % statt 0.50

\definecolor{gelba}{cmyk}{0.00,0.00,0.50,0.00} % statt 0.50

\definecolor{gelbb}{cmyk}{0.00,0.00,0.20,0.00} % statt 0.50

\journalname{European Journal of Mathematics}
\begin{document}

\title{On the Motion of Billiards in Ellipses}


\author{Hellmuth Stachel}


\institute{Vienna Institute of Technology\\
              \email{stachel@dmg.tuwien.ac.at}\\
              \emph{Vienna University of Technology}
}

% \date{Received: date / Accepted: date} % nur fuer arXiv geloescht!

\maketitle

\begin{abstract}
For billiards in an ellipse with an ellipse as caustic, there exist canonical coordinates such that the billiard transformation from vertex to vertex is equivalent to a shift of coordinates.
A kinematic analysis of billiard motions paves the way to an explicit canonical parametrization of the billiard and even of the associated Poncelet grid.
This parametrization uses Jacobian elliptic functions to the numerical eccentricity of the caustic as modulus.
\keywords{billiard \and billiard motion \and confocal conics, elliptic functions}
\subclass{51N20 \and 53A17 \and 33E05, 22E30}
\end{abstract}

%%%%%%%%%%%%%%%%%%%%%%%%%%%%%%%%%%%%%%%%%%%%%%%%%%%%%%%%%%%%%

\section{Introduction}\label{sec:intro}
%       -------------- 
Already for two centuries, billiards in ellipses have attracted the attention of mathematicians, beginning with J.-V.\ Poncelet and A.\ Cayley.
The assertion that one periodic billiard inscribed in an ellipse $e$ and tangent to a confocal ellipse $c$ implies a one-parameter family of such polygons, is known as the standard example of a Poncelet porism.
It was Cayley who derived analytical conditions for such a pair $(e,c)$ of ellipses.

In 2005 S.\ Tabachnikov published a book on billiards from the viewpoint of integrable systems \cite{Tabach}.
The book \cite{DR_Buch} and various papers by V.\ Dragovi\'c and M.\ Radnovi\'c addressed billiards in conics and quadrics within the framework of dynamical systems.

Computer animations carried out by D.\ Reznik, stimulated a new vivid interest on this well studied topic, where algebraic and analytic methods are meeting.  
Originally, Reznik's experiments focused on {\em billiard motions} in ellipses, i.e., on the variation of billiards with a fixed circumscribed ellipse $e$ and a fixed caustic $c\,$.
He published a list of more than 80 numerically detected invariants in \cite{80} and contributed, together with his coauthors R.\ Garcia, J.\ Koiller and M.\ Helman, several proofs.
Other authors like A.\ Akopyan, M.\ Bialy, A.\ Chavez-Caliz, R.\ Schwartz, and S.\ Tabachnikov published several proofs and found more invariants (e.g., in \cite{Ako-Tab,Bialy-Tab,Chavez}).

For a long time, at least since Jacobi's proof of the Poncelet theorem on periodic billiards (see further references in \cite[p.~320]{DR_russ}), it has been wellknown that there is a tight connection between billiards and elliptic functions (note also \cite[Sect.~11.2]{Duistermaat} and \cite{Bobenko}).
On the other hand, S.\ Tabachnikov proved in his book \cite{Tabach} the existence of {\em canonical} parameters on ellipses with the property that the billiard transformation between consecutive vertices $P_i\mapsto P_{i+1}$ of a billiard acts like a shift on the parameters.

\smallskip
The goal of this paper is to prove that Jacobian elliptic functions with the numerical eccentricity of the caustic $c$ as modulus pave the way to canonical coordinates on the ellipse $e\,$. 
This is a consequence of a kinematic analysis of the billiard motion.
It yields an infinitesimal transformation in the plane which preserves a family of confocal ellipses while it permutes the confocal hyperbolas as well as the tangents of the caustic.
Integration results in a group of transformations with a canonical parameter.
In terms of elliptic functions, we also obtain a mapping that sends a square grid together with the diagonals to a Poncelet grid.\footnote{
% +++++++++
Recently, parametrizations of confocal conics in termins of elliptic functions were also presented in \cite{Bobenko}, but not from the viewpoint of billiards.}
% +++++++++
The paper concludes with applying the results of the velocity analysis to a few  invariants of periodic billiards.
These invariants deal mainly with the distances which occur on each side between the contact point with the caustic and the endpoints. 

%%% ----------------------------------------------------------------------
\section{Confocal conics and billiards} % Sect.2
%       --------------------------------------

At the begin, we recall a few properties of confocal conics.
A family of {\em confocal} central conics is given by
\begin{equation}\label{eq:confocal}
  \frac{x^2}{a^2+k} + \frac{y^2}{b^2+k} = 1, \ \mbox{where} \
    k \in \RR \setminus \{-a^2, -b^2\}
\end{equation}
serves as a parameter in the family.
All these conics share the focal points $F_{1,2} = (\pm d,0)$, where $d^2:= a^2-b^2$.

The confocal family sends through each point $P$ outside the common axes of symmetry two orthogonally intersecting conics, one ellipse and one hyperbola \cite[p.~38]{Conics}.
The parameters $(k_e, k_h)$ of these two conics define the {\em elliptic coordinates} of $P$ with
\[  -a^2 < k_h < -b^2 < k_e\,.
\]
If $(x,y)$ are the cartesian coordinates of $P$, then $(k_e,k_h)$ are the roots of the quadratic equation
\begin{equation}\label{eq:cart_in_ell}
  k^2 + (a^2 + b^2 - x^2 - y^2)k + (a^2 b^2 - b^2 x^2 - a^2 y^2) = 0,
\end{equation}
while conversely
\begin{equation}\label{eq:ell_in_cart}
   x^2 = \frac{(a^2 + k_e)(a^2 + k_h)}{d^2}\,, \quad
    y^2 = -\frac{(b^2 + k_e)(b^2 + k_h)}{d^2}\,.
\end{equation}

Suppose that $(a,b)$ in \eqref{eq:confocal} are the semiaxes $(a_c,b_c)$ of the ellipse $c$ with $k = 0$.
Then, for points $P$ on a confocal ellipse $e$ with semiaxes $(a_e,b_e)$ and $k = k_e > 0$, i.e., exterior to $c$, the standard parametrization yields 
\begin{equation}\label{eq:P_coord}
 \begin{array}{c}
   P = (x,y) = (a_e\cos t,\,b_e\sin t), \ 0 \le t < 2\pi,
   \\[1.0mm]
   \mbox{with} \ a_e^2 = a_c^2 + k_e, \ b_e^2 = b_c^2 + k_e\,. 
 \end{array}
\end{equation}
For the elliptic coordinates $(k_e,k_h)$ of $P$ follows from \eqref{eq:cart_in_ell} that 
\[ k_e + k_h = a_e^2\cos^2 t + b_e^2\sin^2 t - a_c^2 - b_c^2.
\]
After introducing the respective tangent vectors of $e$ and $c$, namely
\def\arraycolsep{0.6mm}
\begin{equation}\label{eq:te_und_tc}
 \begin{array}{rcl}
  \Vkt t_e(t) &:= &(-a_e\sin t,\, b_e\cos t), 
  \\[0.8mm]   
  \Vkt t_c(t) &:= &(-a_c\sin t,\, b_c\cos t),
 \end{array}  \ \mbox{where} \
  \Vert \Vkt t_e\Vert^2 = \Vert \Vkt t_c\Vert^2 + k_e\,,   
\end{equation}
we obtain
\begin{equation}\label{eq:k_h}
   k_h = k_h(t) = -(a_c^2\sin^2 t + b_c^2\cos^2 t) = -\Vert\Vkt t_c(t)\Vert^2   
   = -\Vert\Vkt t_e(t)\Vert^2 + k_e
\end{equation}
and $\Vert\Vkt t_e(t)\Vert^2 = k_e - k_h(t)\,$.
Note that points on the confocal ellipses $e$ and $c$ with the same parameter $t$ have the same coordinate $k_h$.
Consequently, they belong to the same confocal hyperbola (\Figref{fig:Poncelet_grid2}).
Conversely, points of $e$ or $c$ on this hyperbola have a parameter out of $\{t, -t, \pi+t, \pi-t\}$.

Let $\theta_i/2$ denote the angle between the tangents drawn from any point $P_i\in e$ to $c$ and the tangent to $e$ at $P_i$ (Figures~\ref{fig:Graves} or \ref{fig:vel_gesamt}).
Then we obtain for $P_i = (a_e\cos t_i,\,b_e\sin t_i)$ with elliptic coordinates $\left( k_e, k_h(t_i)\right)$ 
\begin{equation}\label{eq:Winkel/2} 
 \begin{array}{c}
  \sin^2 \Frac{\theta_i}2 = \Frac{k_e}{\Vert\Vkt t_e(t_i)\Vert^2} 
  = \Frac{k_e}{k_e - k_h(t_i)}, \quad
    \tan\Frac{\theta_i}2 = \pm\sqrt{-\Frac{k_e}{k_h(t_i)}}
   \\[3.0mm]
   \mbox{and}\zwi
   \sin\theta_i = \pm\Frac{2\sqrt{-k_e k_h(t_i)}}{k_e - k_h(t_i)}
   = \pm\Frac{2 \Vert\Vkt t_c(t_i)\Vert \sqrt{k_e}(\Vert\Vkt t_e(t_i)\Vert^2}\,. 
 \end{array}
\end{equation} 
For a proof see \cite{Sta_I}.
We can assume a counter-clockwise order of the billiard.
Hence, all exterior angles are positive.

\begin{figure}[htb] % Fig.1
  \centering
  \def\sz{\small}
  \psfrag{P1}[lb]{\contourlength{1.0pt}\contour{white}{\sz\red $P_1$}}
  \psfrag{P2}[lb]{\contourlength{1.0pt}\contour{white}{\sz\red $P_2$}}
  \psfrag{P3}[rb]{\contourlength{1.0pt}\contour{white}{\sz\red $P_3$}}
  \psfrag{P4}[rb]{\contourlength{1.0pt}\contour{white}{\sz\red $P_4$}}
  \psfrag{P5}[rt]{\contourlength{1.0pt}\contour{white}{\sz\red $P_5$}}
  \psfrag{P6}[rt]{\contourlength{1.0pt}\contour{white}{\sz\red $P_6$}}
  \psfrag{P7}[ct]{\contourlength{1.0pt}\contour{white}{\sz\red $P_7$}}
  \psfrag{P8}[lt]{\contourlength{1.0pt}\contour{white}{\sz\red $P_8$}}
  \psfrag{P9}[lt]{\contourlength{1.0pt}\contour{white}{\sz\red $P_9$}}
  \psfrag{Q1}[rt]{\contourlength{1.0pt}\contour{gelbb}{\sz\blue $Q_1$}}
  \psfrag{Q2}[rt]{\contourlength{1.0pt}\contour{gelbb}{\sz\blue $Q_2$}}
  \psfrag{Q3}[lt]{\contourlength{1.0pt}\contour{gelbb}{\sz\blue $Q_3$}}
  \psfrag{Q4}[lc]{\contourlength{1.0pt}\contour{gelbb}{\sz\blue $Q_4$}}
  \psfrag{Q5}[lb]{\contourlength{1.2pt}\contour{gelbb}{\sz\blue $Q_5$}}
  \psfrag{Q6}[lb]{\contourlength{1.0pt}\contour{gelbb}{\sz\blue $Q_6$}}
  \psfrag{Q7}[rb]{\contourlength{1.0pt}\contour{gelbb}{\sz\blue $Q_7$}}
  \psfrag{Q8}[rb]{\contourlength{1.0pt}\contour{gelbb}{\sz\blue $Q_8$}}
  \psfrag{Q9}[rc]{\contourlength{1.0pt}\contour{gelbb}{\sz\blue $Q_9$}}
  \psfrag{S1}[lb]{\contournumber{32}\contourlength{1.2pt}\contour{white}{\sz\dgrau $S_1^{(1)}$}}
  \psfrag{S2}[lb]{\contourlength{1.2pt}\contour{white}{\sz\dgrau $S_2^{(1)}$}}
  \psfrag{S3}[lb]{\contournumber{32}\contourlength{1.2pt}\contour{white}{\sz\dgrau $S_3^{(1)}$}}
  \psfrag{S4}[rc]{\contournumber{32}\contourlength{1.2pt}\contour{white}{\sz\dgrau $S_4^{(1)}$}}
  \psfrag{S5}[rt]{\contourlength{1.2pt}\contour{white}{\sz\dgrau $S_5^{(1)}$}}
  \psfrag{S6}[rt]{\contournumber{32}\contourlength{1.2pt}\contour{white}{\sz\dgrau $S_6^{(1)}$}}
  \psfrag{S7}[lt]{\contourlength{1.2pt}\contour{white}{\sz\dgrau $S_7^{(1)}$}}
  \psfrag{S8}[lt]{\contourlength{1.2pt}\contour{white}{\sz\dgrau $S_8^{(1)}$}}
  \psfrag{S9}[lc]{\contournumber{32}\contourlength{1.2pt}\contour{white}{\sz\dgrau $S_9^{(1)}$}}
  \psfrag{S12}[lt]{\contourlength{1.0pt}\contour{white}{\sz\green $S_1^{(2)}$}}
  \psfrag{S42}[lb]{\contourlength{1.0pt}\contour{white}{\sz\green $S_4^{(2)}$}}
  \psfrag{S52}[ct]{\contourlength{1.0pt}\contour{white}{\sz\green $S_5^{(2)}$}}
  \psfrag{S92}[ct]{\contourlength{1.0pt}\contour{white}{\sz\green $S_9^{(2)}$}}
  \psfrag{S82}[lc]{\contourlength{1.0pt}\contour{white}{\sz\green $S_8^{(2)}$}}
  \psfrag{c}[rt]{\blue $\boldsymbol{c}$}
  \psfrag{e}[lb]{\red $\boldsymbol{e}$}  \psfrag{e1}[lb]{\green $\boldsymbol{e}^{(1)}$}
  \psfrag{e2}[rt]{\green $\boldsymbol{e}^{(2)}$} % 115mm
  \includegraphics[width=110mm]{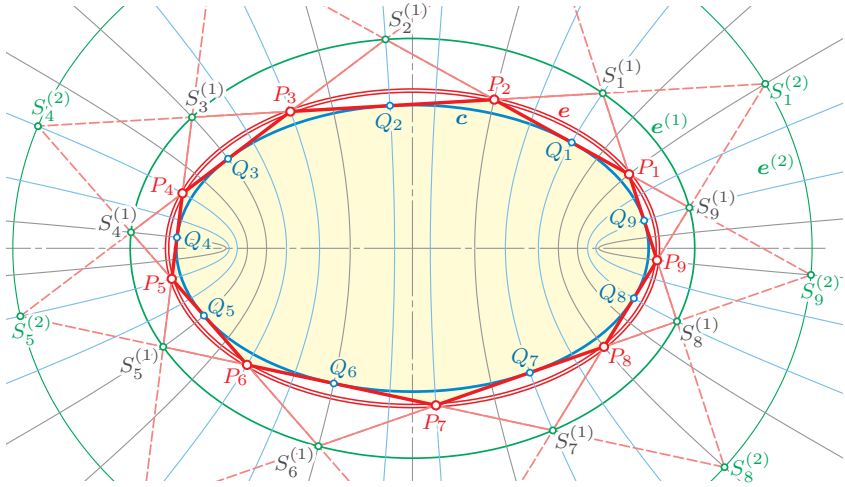} % biliard8.pas
  \caption{Periodic billiard $P_1P_2\dots P_9$ in $e$ with the turning number $\tau\!=\!1$ and the caustic $c$ as well as a part of the associated Poncelet grid. 
 The extended sides form a billiard with $\tau=2$ in $e^{(1)}$ and three triangles as billiards in $e^{(2)}$.}
  \label{fig:Poncelet_grid2}
\end{figure}

\smallskip
From \eqref{eq:k_h} follows
\[  k_h = -\frac{a_c^2\tan^2 t + b_c^2}{1 + \tan^2 t}, \quad\mbox{hence}\quad
  \tan^2 t(a_c^2 + k_h) = -b_c^2 - k_h
\]
and furthermore
\begin{equation}\label{eq:st_ct}
  \sin t\cos t = \frac{\tan t}{1 + \tan^2 t} 
    = \frac{\sqrt{-(b_c^2 + k_h)(a_c^2 + k_h)}}{a_c^2 - b_c^2}
    = \frac{a_h\, b_h}{d^2}                          
%   \tan^2 t = -\frac{b_c^2 + k_h(t)}{a_c^2 + k_h(t)} \quad\mbox{and}\quad
\end{equation}   
with $a_h$ and $b_h$ as semiaxes of the hyperbola corresponding to the parameter $t$, i.e., $a_h^2 = a_c^2 + k_h$ and $b_h^2 = -(b_c^2 + k_h)$.

\medskip
Let $\dots P_1P_2P_3\dots$ be a billiard in the ellipse $e$. 
Then the extended sides intersect at points 
\begin{equation}\label{eq:S_i^j}
  S_i^{(j)}:= \left\{ \begin{array}{rl}
  [P_{i-k-1},P_{i-k}]\cap[P_{i+k},P_{i+k+1}] &\zwi\mbox{for} \ j = 2k,
   \\[0.6mm] 
  [P_{i-k},P_{i-k+1}]\cap[P_{i+k},P_{i+k+1}] &\zwi\mbox{for} \ j = 2k-1,
 \end{array} \right.  
\end{equation}
where $i = \dots,1,2,3,\dots$ and $j= 1,2,\dots$.
These points are distributed on different confocal conics:
For fixed $j$, there are ellipses $e^{(j)}$ passing through the points $S_i^{(j)}$.
On the other hand, the points $S_i^{(2)}$, $S_i^{(4)},\dots$ are located on the confocal hyperbola through $P_i$, while $S_i^{(1)}$, $S_i^{(3)}, \dots$ belong to the confocal hyperbola through the contact point $Q_i$ between the side $P_i P_{i+1}$ and the caustic $c\,$.
This configuration is called the associated {\em Poncelet grid} (\Figref{fig:Poncelet_grid2}).
For periodic billiards the sets of points $S_i^{(j)}$ and associated conics are finite.
The {\em turning number} $\tau$ of a periodic billiard in $e$ with an ellipse as caustic counts how often one period of the billiard surrounds the center $O$ of $e$ (note \Figref{fig:Poncelet_grid2}).

For each billiard $P_1P_2\dots$ in $e$ with caustic $c$ there exists a {\em conjugate billiard} $P_1'P_2'\dots$ in $e$ with the same caustic.
The axial scaling $c\to e$ maps the contact point $Q_i\in c$ of $P_iP_{i+1}$ to $P_i'$ while the inverse scaling brings $P_i$ to the contact point $Q_{i-1}'$ of $P_{i-1}'P_i'$ with the caustic.
The relation between these billiards is symmetric.
For further details see \cite[Sect.~3.2]{Sta_I}.

% -----------------------------------------------------------------------------
\section{Velocity analysis}
%       ------------------- 

Let the first vertex of a billiard $P_1P_2\dots $ move smoothly along the circumscribed ellipse $e$.
Then this induces a continuous variation of all other vertices along $e$ and also of the intersection points $S_i^{(j)}$ along $e^{(j)}$.
We call this a {\em billiard motion}, though it neither preserves angles or distances nor is an affine or projective motion. 

According to Graves's construction \cite[p.~47]{Conics}, we can conceive the periodic billiard $P_1P_2\dots P_N$ as a flexible chain of fixed total length $L_e$ and the caustic $c$ as a fixed non-circular chain wheel.
The vertices $P_1,P_2,\dots$ move along $e$ and relative to the chain such that they keep the chain strengthened, while the chain contacts $c$ only at the single points $Q_1,Q_2,\dots,Q_N$. 

Let us pick out a single vertex $P_2$ (see \Figref{fig:Graves}).
In the language of kinematics, the line spanned by the straight segment $Q_1P_2$ rolls at $Q_1$ on $c$ (=\,fixed polode) while point $P_2$ moves along the line (=\,moving polode) with the velocity vector $\Vkt v_{t_1}$ .
The instantaneous rotation about $Q_1$ with the angular velocity $\omega_1$ assigns to $P_2$ a velocity vector $\Vkt v_{n_1}$ orthogonal to $Q_1P_2$ in order to keep the vector of absolute velocity of $P_2\,$, namely $\Vkt v_2 = \Vkt v_{t_1} + \Vkt v_{n_1}$, tangent to the ellipse $e$.

\begin{figure}[hbt] % Fig.2
  \centering % ursprungliches Bild an y-Achse gespiegelt
  \psfrag{F1}[rt]{\sz $F_2$}
  \psfrag{F2}[lt]{\sz $F_1$}
  \psfrag{X}[lb]{\sz\red $P_2$}
  \psfrag{X}[lb]{\sz\red $P_2$}
  \psfrag{tP}[lc]{\sz $t_P$}
  \psfrag{T1}[lb]{\sz\blue $Q_1$}
  \psfrag{T2}[rb]{\sz\blue $Q_2$}
%   \psfrag{c0}[rt]{\sz\blue $\Vkt c(0)$}
  \psfrag{la}[cb]{\sz\blue $r_2$}
  \psfrag{l2}[lt]{\sz\blue $l_2$}
  \psfrag{e}[lb]{\red $\boldsymbol{e}$}
  \psfrag{e0}[rt]{\blue $\boldsymbol{c}$}
  \psfrag{vr1}[rt]{\sz $\Vkt v_{t_1}$}
  \psfrag{vn1}[rb]{\sz $\Vkt v_{n_1}$}
  \psfrag{vr2}[rb]{\sz $\Vkt v_{t_2}$}
  \psfrag{vn2}[rt]{\sz $\Vkt v_{n_2}$}
  \psfrag{v}[rb]{\sz $\Vkt v_2$}
  \psfrag{om1}[ct]{\sz $\omega_1$}
  \psfrag{om2}[lt]{\sz $\omega_2$}
  \psfrag{theta}[cc]{\blue $\frac{\theta_2}2$} % 95mm
  \includegraphics[width=80mm]{\pfad 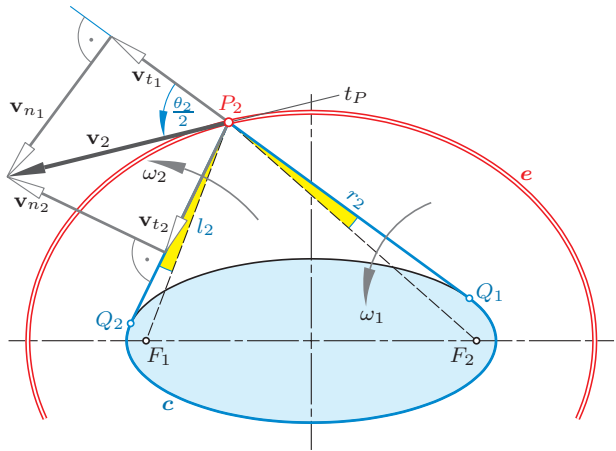} % CAD_2D: ell_krue
  \caption{{Graves}'s string construction of an ellipse $e$ confocal to $c$.}
  \label{fig:Graves}
\end{figure}

Similarly, we have a second decomposition $\Vkt v_2 = \Vkt v_{t_2} + \Vkt v_{n_2}$, since at the same time the line $[Q_2,P_2]$ rotates about $Q_2$ with the angular velocity $\omega_2$, while $P_2$ moves relative to this line.
Due to the constant length of the chain, the tangential components in these two decompositions must be of equal lengths $\Vert\Vkt v_{t_2}\Vert = \Vert\Vkt v_{t_1}\Vert$.
Since the tangent $t_P$ to $e$ at $P_2$ bisects the exterior angle of $Q_1P_2Q_2$, the second decomposition is symmetric w.r.t.\ $t_P$ to the first one.
From $\Vert\Vkt v_{n_2}\Vert = \Vert\Vkt v_{n_1}\Vert$ follows for the distances $r_2:= \ol{P_2Q_1}$ and $l_2:= \ol{P_2Q_2}$
\begin{equation}\label{eq:vn}
   l_2\,\omega_2 = r_2\,\omega_1, \quad \mbox{or} \quad
    \frac{\omega_1}{\omega_2} = \frac{l_2}{r_2},
\end{equation}
and similarly for all other vertices.
If the billiard is $N$-periodic, then the product of all ratios $l_i/r_i$ for $i=1,\dots,N$ yields
\[ \frac{l_1}{r_1}\cdot \frac{l_2}{r_2} \cdots \frac{l_N}{r_N} = 
    \frac{\omega_N}{\omega_1}\cdot \frac{\omega_1}{\omega_2}\cdots
    \frac{\omega_{N-1}}{\omega_N} = 1,
\]   
which results in the equation 
\begin{equation}\label{eq:r1..rN}  
  l_1 l_2\dots l_N = r_1 r_2\dots r_N.
\end{equation}
listed as k116 in \cite[Table~2]{80}.

Figure~\ref{fig:vel_gesamt} shows a graphical velocity analysis for the billiard motion of a 5-sided periodic billiard in $e$.
We can begin this analysis by choosing an arbitrary length for the arrow representing the velocity vector $\Vkt v_2$ of $P_2$.
This defines the two components $\Vkt v_{t_2}$ and $\Vkt v_{n_2}$, where the latter determines the angular velocity $\omega_2$ of the side $P_2P_3$ and furtheron the absolute velocity $\Vkt v_3$ of $P_3$.
This can be continued.
From now on, we denote the norms $\Vert\Vkt v_{t_1}\Vert = \Vert\Vkt v_{t_2}\Vert$ and $\Vert\Vkt v_{n_1}\Vert = \Vert\Vkt v_{n_2}\Vert$ of the respective components of the velocity vector $\Vkt v_i$ of $P_i$ with $v_{t|i}$ and $v_{n|i}$.

\begin{figure}[hbt] % Fig.3
  \centering %  ab nun gespiegelt an y-Achse
  \psfrag{P1}[lb]{\sz\red $P_1$}
  \psfrag{P2}[lb]{\sz\red $P_2$}
  \psfrag{P3}[rb]{\sz\red $P_3$}
  \psfrag{P4}[rt]{\sz\red $P_4$}
  \psfrag{P5}[ct]{\sz\red $P_5$} % bis hierher geaendert!
  \psfrag{T1}[rt]{\sz\blue $Q_1$}
  \psfrag{T2}[lt]{\sz\blue $Q_2$}
  \psfrag{T3}[lb]{\sz\blue $Q_3$}
  \psfrag{T4}[cb]{\sz\blue $Q_4$}
  \psfrag{T5}[rc]{\sz\blue $Q_5$}
  \psfrag{Q1}[lc]{\sz\red $R_1$}
  \psfrag{Q2}[rc]{\sz\red $R_2$}
  \psfrag{Q3}[rt]{\sz\red $R_3$}
  \psfrag{Q4}[rt]{\sz\red $R_4$}
  \psfrag{Q5}[lc]{\sz\red $R_5$}
%   \psfrag{O}[rt]{\sz $O$}
  \psfrag{c}[lb]{\red $\boldsymbol{e}$}
  \psfrag{c_}[lb]{\blue $\boldsymbol{c}$}
%   \psfrag{c_}[rb]{\sz $c'$}
  \psfrag{l1}[lb]{\sz\red $l_1$} % r - l vertauscht!
  \psfrag{r1}[lb]{\sz\red $r_2$}
  \psfrag{l2}[rb]{\sz\red $l_2$}
  \psfrag{r2}[rb]{\sz\red $r_3$}
  \psfrag{vr1}[rt]{\sz\red $\Vkt v_{t_1}$}
  \psfrag{vn1}[rb]{\sz\red $\Vkt v_{n_1}$}
  \psfrag{vr2}[rb]{\sz\red $\Vkt v_{t_2}$}
  \psfrag{vn2}[rt]{\sz\red $\Vkt v_{n_2}$}
  \psfrag{v}[rb]{\sz\red $\Vkt v$}
  \psfrag{the1o}[rb]{\sz\green $\theta_1\mskip -2mu/2$}
  \psfrag{the1u}[lc]{\sz\green $\theta_1\mskip -2mu/2$}
  \psfrag{the2r}[rc]{\contourlength{1.6pt}\contour{white}{\sz\green $\theta_2\mskip -1mu/2$}}
  \psfrag{the2l}[cb]{\sz\green $\theta_2\mskip -1mu/2$}
  \psfrag{om1}[ct]{\sz\blue $\omega_1$}
  \psfrag{om2}[lc]{\sz\blue $\omega_2$}
  \psfrag{om3}[lt]{\contourlength{1.5pt}\contour{gelba}{\sz\blue $\omega_3$}}
  \psfrag{om4}[rc]{\sz\blue $\omega_4$}
  \psfrag{om5}[rb]{\contourlength{1.5pt}\contour{gelba}{\sz\blue $\omega_5$}}
  \psfrag{v2}[lt]{\contourlength{1.6pt}\contour{white}{\sz $\Vkt v_2$}}
  \psfrag{vt2}[lb]{\sz $\Vkt v_{t_2}$}
  \psfrag{vn2}[rb]{\sz $v_{n|2}$} % \!=\!r_2\omega_1\!=\!l_2\omega_2$}
  \psfrag{v1}[lb]{\contourlength{1.6pt}\contour{white}{\sz $\Vkt v_1$}}
  \psfrag{v3}[rb]{\contourlength{1.6pt}\contour{white}{\sz $\Vkt v_3$}}
  \psfrag{v4}[rt]{\contourlength{1.6pt}\contour{white}{\sz $\Vkt v_4$}}
  \psfrag{v5}[lt]{\contourlength{1.6pt}\contour{white}{\sz $\Vkt v_5$}}
  \includegraphics[width=105mm]{\pfad 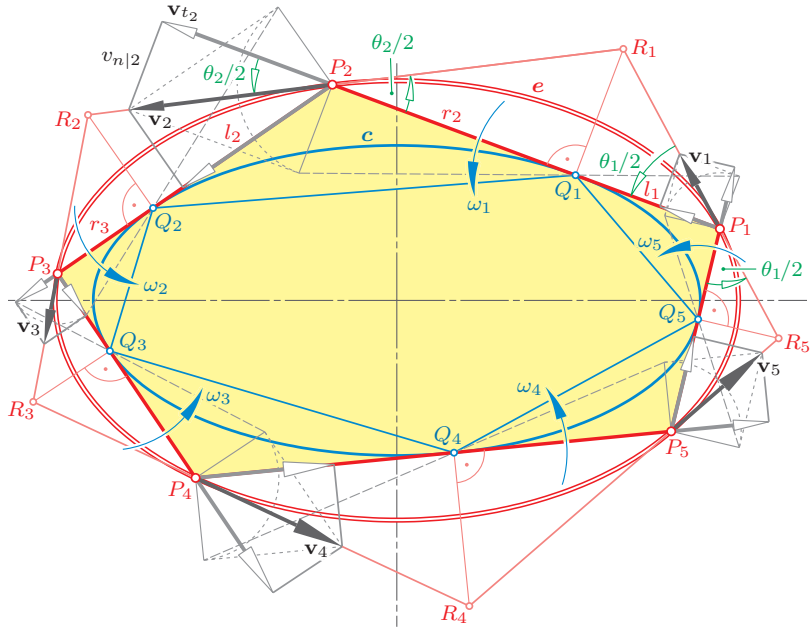} % 130mm
  \caption{Velocities of the vertices $P_1, P_2, \dots, P_5$ of a periodic billiard in the ellipse $e$ with the caustic $c\,$.}
  \label{fig:vel_gesamt}
\end{figure}

In terms of the exterior angles $\theta_1,\dots,\theta_N$ of the billiard, we obtain from \eqref{eq:vn}
\begin{equation}\label{eq:sin,cos_theta}
   \sin\Frac{\theta_2}2 = \Frac{l_2\,\omega_2}{v_2} = \Frac{r_2\,\omega_1}{v_2} 
   \zwi \mbox{and} \zwi \cos\Frac{\theta_2}2 = \Frac{v_{t|2}}{v_2}\,,\zwi
   \mbox{where} \zwi v_2:= \Vert\Vkt v_2\Vert.
\end{equation}
Let $R_i$ denote the pole of the line $[P_i,P_{i+1}]$ with respect to (w.r.t.\ in brief) $e\,$.
Since the poles of a line $\ell$ w.r.t.\ confocal conics lie on a line orthogonal to $\ell$, the side $P_1P_2$ is orthogonal to $[Q_1,R_1]$ (\Figref{fig:vel_gesamt}), which means
\begin{equation}\label{eq:tan_theta}
  \ol{R_1Q_1} = l_1\tan\frac{\theta_1}2 = r_2\tan\frac{\theta_2}2\,.
\end{equation}
From \eqref{eq:tan_theta} and \eqref{eq:sin,cos_theta} follows
\[  l_1\tan\frac{\theta_1}2 = l_1\frac{l_1\omega_1}{v_{t|1}} 
    = r_2\tan\frac{\theta_1}2 = r_2\frac{r_2\omega_1}{v_{t|2}}
   \zwi\mbox{and}\zwi
   \frac{v_{t|2}}{v_{t|1}} = \frac{r_2^2}{l_1^2} 
    = \frac{\tan^2\frac{\theta_1}2}{\tan^2\frac{\theta_2}2}\,.  
\]
This shows that, by virtue of \eqref{eq:Winkel/2}, the products
\begin{equation}\label{eq:vers2_vt1:vt2}
   v_{t|1}\tan^2\Frac{\theta_1}2 = v_{t|2}\tan^2\Frac{\theta_2}2 = \dots 
   = v_{t|N}\tan^2\Frac{\theta_N}2 = v_{t|i}\Frac{k_e}{\Vert\Vkt t_{c|i}\Vert^2} 
\end{equation}
for $i=1,2,\dots,N$ are equal along the billiard.
We denote this quantity with $C$. % C statt K, neu fuer v2
Instead of a free choice of $v_2\,$, it means no restriction of generality to set
$C = k_e\,$.
Then we obtain for the point $P_i = (a_e\cos t_i,\,b_e\sin t_i)$ of the ellipse $e$, by virtue of \eqref{eq:Winkel/2},
\begin{equation}\label{eq:vt, v global}
 \begin{array}{c}
   v_{t|i} = \Vert\Vkt t_c\Vert^2 = -k_h, \quad
    v_{n|i} = v_i \sin\smFrac{\theta}2 =  \Vert\Vkt t_c\Vert\,\sqrt{k_e} = \sqrt{-k_e k_h} 
  \\[2.0mm]
   v_i = \Frac{\Vert\Vkt t_c\Vert^2}{\cos\frac{\theta}2} 
      = \Vert\Vkt t_c\Vert\,\Vert\Vkt t_e\Vert = \sqrt{k_h(k_h - k_e)}\zwi
   \mbox{for}\zwi t = t_i\zwi\mbox{and}\zwi k_h = k_h(t_i)\,. 
 \end{array}
\end{equation}

% -----------------------------------------------------------------------------
\section{Billiard motion and the underlying Lie group}
%       ----------------------------------------------

Our specification of the quantity $C$ assigns to the vertex $P_i\in e$ with parameter $t_i$ a non-vanishing velocity vector $\Vkt v_i = \Vert\Vkt t_c(t_i)\Vert\,\Vkt t_e(t_i)\,$.
This assignment can immediately be extended to all points of $e\,$.
There exists a parameter $u$ on $e$ such that the differentiation by $u$ gives the said velocity vector.
Let a dot indicate this differentiation.
Then
\begin{equation}\label{eq:v von t}
  \Vkt v(t) = \dotVkt p(t) = \Frac{\mathrm d\,\Vkt p(t)}{\mathrm d u} 
   = \Vert\Vkt t_c(t)\Vert\,\Vkt t_e(t) %= \sqrt{-k_h(t)}
   = \sqrt{a_c^2\sin^2 t + b_c^2\cos^2 t} \ \Vkt t_e(t).
\end{equation}
We can even extend this to all confocal ellipses of $c\,$.
The assignment of a velocity vector
\[ \Vkt v(t) = \Vert\Vkt t_c(t)\Vert\,\Vkt t_e(t) = \dot t\;\Vkt t_e(t)
\]
to each point $P = (a_e\cos t,\,b_e\sin t)$ with $a_e^2 - b_e^2 = a_c^2 - b_c^2$  defines an `instant motion' of the plane, where
\begin{equation}\label{eq:dot t} 
   \dot t = \frac{\mathrm d t}{\mathrm d u} = \Vert\Vkt t_c(t)\Vert
    = \sqrt{-k_h(t)} = \sqrt{a_c^2\sin^2 t + b_c^2\cos^2 t}\,.
\end{equation}

\begin{figure}[htb] % Fig.4
  \centering
  \psfrag{c}[lt]{\blue $\boldsymbol{c}$}
  \psfrag{e}[lt]{\red $\boldsymbol{e}$}
  \psfrag{Q}[lt]{\contourlength{1.5pt}\contour{white}{\sz\blue $Q$}}
  \psfrag{tQ}[rb]{\sz $t_Q$}
  \psfrag{vn}[lt]{\sz\red $\Vkt v_n$}
  \psfrag{vt}[lt]{\sz\red $\Vkt v_t$} % 120mm
  \includegraphics[width=90mm]{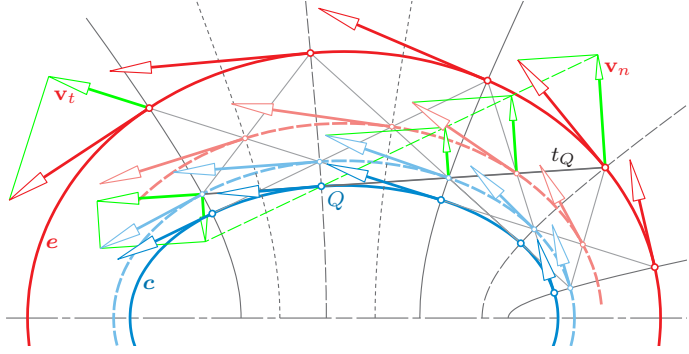} % izmest5.pas mit Geschwindigkeiten
  \caption{The infinitesimal motion assigns to each point of the Poncelet grid a velocity vector such that the points on each tangent $t_Q$ to the caustic $c$ remain aligned.}
  \label{fig:vel_field} % aehnlich zu izmest_tangente2!
\end{figure}

We prove below, that this `instant motion' is compatible with the billiard and the associated Poncelet grid.
This means in particular, that the velocity \eqref{eq:vt, v global} is also valid for all points $S_i^{(j)}\in e^{(j)}$.

\Figref{fig:vel_field} shows a portion of the Poncelet grid and the velocity vectors of a couple of points, each represented by a scaled arrow.
As indicated, for all points on the tangent $t_Q$ to $c$ at any point $Q\in c$, the normal components $\Vert \Vkt v_n\Vert$ of the respective velocity vectors $\Vkt v$ are proportional to the distance to $Q$.
On the other hand, all points on any confocal hyperbola share the tangential component $\Vert \Vkt v_t\Vert$.

\begin{thm}\label{thm:vel} % Thm.1
Let the billiard $P_1P_2\dots$ with the ellipse $c$ as caustic be moving along the circumscribed ellipse $e\,$.
Then the motion is the action of a one-parameter Lie group $\Gamma$.
Each transformation $\gamma(u)\in\Gamma$ preserves the confocal ellipses and permutes the confocal hyperbolas as well as the tangents to $c\,$.
\begin{enumerate}
\item If $(a_c,b_c)$ are the semiaxes of the caustic $c$ with the tangent vectors $\Vkt t_c(t) = (-a_c\sin t,\,b_c\cos t)$, then for all confocal ellipses $e$ with semiaxes $(a_e,b_e)$ the $\Gamma$ generating instant motion is defined, up to a scalar, by the vector field
\begin{equation}\label{eq:inf_motion1}
  (x,y) = (a_e\cos t,\,b_e\sin t)\, \ \mapsto\, \Vert\Vkt t_c\Vert \,\Vkt t_e
  = \sqrt{-k_h(t)}\,\left(-\Frac{a_e y}{b_e}, \,\Frac{b_e x}{a_e}\right)
\end{equation}
with $a_e^2 - b_e^2 = a_c^2 - b_c^2 = d^2$ and $a_e^2 - a_c^2 \ge 0$. 

\smallskip
\item If we parametrize the quadrant $x,y > 0$ by elliptic coordinates as $\Vkt X(k_e,k_h)$, then the vector field can be expressed as
\begin{equation}\label{eq:inf_motion2}
  \Vkt X(k_e,\,k_h) \ \mapsto \ -2 \sqrt{k_h(a_c^2 + k_h)(b_c^2 + k_h)}\;
  \frac{\partial\Vkt X}{\partial\, k_h}\,.
\end{equation}
\end{enumerate}
\end{thm}

The vector field defines a canonical parameter $u$ for the one-parameter Lie group $\Gamma$, i.e., for transformations $\gamma(u)\in\Gamma$ holds $\gamma(u_2)\circ\gamma(u_1) = \gamma(u_1 + u_2)$.
At the same time, $u$ provides canonical coordinates\footnote{
% ++++++++++
Of course, we still obtain canonical coordinates on the ellipses, when all velocity vectors are multiplied with any constant $\lambda\in \mathbb R \setminus{0}$.}
% ++++++++++
 on each confocal ellipse with the property % (\Figref{fig:tangente2}) 
\[  \gamma(2\varDelta u)\!: \ P_i \mapsto P_{i+1}, \ S_i^{(j)} \mapsto S_{i+1}^{(j)}. 
\]

\begin{proof}
1. The first derivative $\dot t$ in \eqref{eq:dot t} is independent of the choice of the ellipse $e$.
Therefore $\gamma(u)$ permutes the confocal hyperboloids.
On the other hand, the representation $\Vkt v = \Vert\Vkt t_c\Vert \,\Vkt t_e$ reveals that all confocal ellipses remain fixed.
Furthermore, we verify that the position of any point $P$ on the tangent $t_Q$ to $c$ at $Q$ (see \Figref{fig:vel_field}) is preserved under the infinitesimal motion:

\noindent
Given $P = \left(a_e\cos t, \,b_e\sin t\right)\in e$ and $Q = 
\left(a_c\cos t', \,b_c\sin t'\right)$, the point $P$ lies on $t_Q$ if and only if 
\begin{equation}\label{eq:P in tQ}
  b_c a_e\cos t'\cos t + a_c b_e\sin t'\sin t = a_c b_c\,.
\end{equation}
This is preserved under the infinitesimal motion if differentiation by $u$ based on \eqref{eq:dot t} yields an identity, namely
\def\arraycolsep{0.8mm}
\begin{equation}\label{eq:Abl_P in tT}
 \begin{array}{l}
   \phantom{-}\Vert\Vkt t_c(t')\Vert \left(-b_c a_e\sin t'\cos t + a_c b_e\cos t'\sin t\right)\phantom{.}
  \\[1.5mm]
  = -\Vert\Vkt t_c(t)\Vert \left(-b_c a_e\cos t'\sin t 
  + a_c b_e\sin t'\cos t \right). 
 \end{array}
\end{equation}
In order to verify this, we square both sides and substitute from the squared equation \eqref{eq:P in tQ} the mixed term $2 a_c b_c a_e b_e \sin t' \cos t' \sin t \cos t\,$.
After some computations, this yields for both sides
\[  d^2\left(\sin^2 t - \sin^2 t'\right)\left(a_c^2 b_e^2\sin^2 t'\sin^2 t 
   + b_c^2 a_e^2 \cos^2 t'\cos^2 t - a_c^2 b_c^2\right).
\] 
The velocity analysis in \eqref{eq:vt, v global} for the particular ellipse $e$ confirms, that also the signs of both sides in \eqref{eq:Abl_P in tT} are equal.

\medskip\noindent
2. From \eqref{eq:ell_in_cart} follows for $(x,y) = \Vkt X(k_e,\,k_h)$
\[  2x\,\frac{\partial x}{\partial k_h} = \frac{a_c^2 + k_e}{d^2}\,, \quad
    2y\,\frac{\partial y}{\partial k_h} = -\frac{b_c^2 + k_e}{d^2}
\]
and therefore
\[  \Vkt X_{k_h} = \frac{\partial\Vkt X}{\partial\, k_h} 
    = \frac 1{2d^2}\left(\frac{a_c^2+k_e}x, \,-\frac{b_c^2+k_e}y\right)
    = \frac{-1}{2d^2\sin t\cos t}\,\Vkt t_e\,.
\]
This implies by \eqref{eq:k_h} and \eqref{eq:st_ct}
\[ \Vert\Vkt t_c\Vert \,\Vkt t_e = \lambda \Vkt X_{k_h}
   \zwi\mbox{with}\zwi
   \lambda % -\Frac{2d^2}{a_eb_e}\, xy \sqrt{-k_h} 
    = -2 a_h b_h \sqrt{-k_h}  
    = -2\,\sqrt{k_h(a_c^2 + k_h)(b_c^2 + k_h)}\,, 
\]     
which confirms the claim in \eqref{eq:inf_motion2}.
\end{proof}

\medskip
In order to express the action of the transformation $\gamma(u)\in\Gamma$ on an initial point $(a_e\cos t,\,b_e\sin t)$, we integrate the differential equation \eqref{eq:dot t}
\[ \dot t = \Frac{\mathrm d t}{\mathrm d u} 
%                = \sqrt{a_c^2\sin^2 t + b_c^2\cos^2 t}
    = \sqrt{a_c^2\sin^2 t + (a_c^2 - d^2)\cos^2 t} = a_c \sqrt{1 - m^2\cos^2 t}
\]
with $m := d/a_c < 1$ as numerical eccentricity of the caustic $c$.
The substitution
\[  \varphi:= t - \smFrac{\pi}2
\]
results in
\[  \frac{\mathrm d\varphi}{\sqrt{1 - m^2\sin^2\varphi}} = a_c\,\mathrm d u\,.
\]
The initial condition $\varphi = 0$ for $u=0$ yields the unique solution
\begin{equation}\label{eq:ell_int}
   a_c\,u(\varphi) = F(\varphi, m) 
    = \int_0^{\varphi} \frac{\mathrm d\varphi}{\sqrt{1 - m^2\sin^2\varphi}}
\end{equation}
with $F(\varphi, m)$ as the elliptic integral of the first kind with the {\em modulus} $m$.
The equation~\eqref{eq:ell_int} shows the canonical coordinate $u$ in terms of $\varphi$ with the {\em quarter period} 
\[  K:= a_c\,u \left(\smFrac{\pi}2\right) = \int_0^{\pi/2} 
    \frac{\mathrm d\varphi}{\sqrt{1 - m^2\sin^2\varphi}} 
%     = \frac{\pi}{2 a_c}\left[1 + \left(\smFrac{1}{2}\right)^{\!\!2} m^2 
%     + \left(\smFrac{1\cdot 3}{2\cdot 4}\right)^{\!\!2} m^4 + \cdots \right].
\]    
For the sake of simplicity, we introduce 
\begin{equation}\label{eq:wt_u}
  \wt u(\varphi):= a_c\,u(\varphi)
\end{equation}
as a new canonical coordinate.

The inverse function of $\wt u = F(\varphi, m)$, namely the Jacobian {\em amplitude} $\varphi = \mathrm{am}\mskip 1mu(\wt u)$ leads to the Jacobian elliptic functions, the {\em elliptic sine} % <https://dlmf.nist.gov/22.3> <https://www.google.com/search?client=firefox-b-d&q=Taylor+series+elliptic+functions>
\[  \SN\wt u = \sin(\mathrm{am}(\wt u)) = \sin\varphi = -\cos t
\]
with $\Sn(-\wt u) = -\SN\wt u\,$, the {\em elliptic cosine}
\[ \CN\wt u = \cos(\mathrm{am}(\wt u)) = \cos\varphi = \sin t 
\]
with $\Cn(-\wt u) = \CN\wt u\,$, and the {\em delta amplitude}
\[  \DN\wt u = \sqrt{1 - m^2 \,\mathrm{sn}^2\wt u}
\]
with $\Dn(-\wt u) = \DN\wt u\,$ as the third elliptic base function \cite{Hoppe}.
Moreover, for $k\in\mathbb Z$ holds
\[ \begin{array}{c}
    \Sn(\wt u + 2kK) = (-1)^k \SN\wt u\,, \quad
    \Cn(\wt u + 2kK) = (-1)^k \CN\wt u\,,
    \\[1mm]
    \Dn(\wt u + 2kK) = (-1)^k \DN\wt u\,.
  \end{array}
\]

This gives rise to the canonical parametrization of the ellipse $e$ with semiaxes $(a_e,b_e)$ as
\[  \left(-a_e\SN\wt u,\,b_e\CN\wt u \right) \quad
    \mbox{for}\zwi 0 \le \wt u < 4K = 4\wt u\left(\ssmFrac{\pi}2\right).
\] % https://dlmf.nist.gov/22

\noindent
As an alternative, we can proceed with elliptic coordinates.
From \eqref{eq:inf_motion2} and 
\[  \frac{\mathrm d\Vkt X}{\mathrm du} 
    = \dot k_e\,\frac{\partial\Vkt X}{\partial\, k_e} + 
    \dot k_h\,\frac{\partial\Vkt X}{\partial\, k_h} = -2\,\sqrt{
    k_h(a_c^2 + k_h)(b_c^2 + k_h)}\;\frac{\partial\Vkt X}{\partial\, k_h}
\]
follows for the orbits of the Lie group $\dot k_e = 0$ and
\[ \dot k_h = -2\,\sqrt{k_h(a_c^2 + k_h)(b_c^2 + k_h)}\,.
    % = -2 a_h b_h\,\sqrt{-k_h}\,.
\]
As expected, the orbits are confocal ellipses.
% , but this time parametrized by the canonical coordinate $u$.
In order to express the action of $\gamma(u)\in\Gamma$ on an initial point $\Vkt X(k_{e|0},\,k_{h|0})$, we have to integrate the differential equation 
\begin{equation}\label{eq:diff_gl_kh}
  \frac{\mathrm dk_h}{\sqrt{k_h(a_c^2 + k_h)(b_c^2 + k_h)}} = -2\,\mathrm du
\end{equation}
with any initial condition. % $k_h(0) = k_{h|0}$.
Again, we face an elliptic integral, this time in the socalled Riemannian form.

\begin{thm}\label{eq:action} % Thm.2
\begin{enumerate} % enumerate neu in v2
\item Let $c$ be the ellipse $c$ with semiaxes $(a_c,b_c)$ and linear eccentricity $d = \sqrt{a_c^2 - b_c^2}\,$.
Then for all confocal ellipses $e$ with semiaxes $(a_e,b_e)$, the inscribed billiards with the  caustic $c$ can be canonically parametrized using the Jacobian elliptic functions to the modulus $m = d/a_c$ (=\,numerical eccentricity of $c$) as
\[  \left(-a_e\SN\wt u,\ b_e\CN\wt u \right).
\] % https://dlmf.nist.gov/22
This means that, if $b_c = b_e\Cn(\varDelta\wt u)$, then the vertices of the billiards in $e$ have the canonical parameters $\wt u = (\wt u_0 + 2k\varDelta\wt u)$ for $k\in\mathbb Z$ and any given initial $\wt u_0\,$.
\item Conversely, we obtain an ellipse $e$ for which the billiards with caustic $c$ are $N$-periodic with turning number $\tau$, where $\mathrm{gcd}(N,\tau) = 1$, by the choice
\[  \varDelta\wt u = \frac{2\tau K}N
\]
with $K$ as the complete elliptic integral of the first kind to the modulus $m$, provided that 
\begin{equation}\label{eq:e zu Delta_u}
    a_e = \frac{a_c\Dn(\varDelta\wt u)}{\Cn(\varDelta\wt u)}\zwi\mbox{and}\quad 
    b_e = \frac{b_c}{\Cn(\varDelta\wt u)}\,.
\end{equation}
\end{enumerate}
\end{thm}    

\begin{figure}[htb] % Fig.5
  \centering
  \psfrag{Q1}[lt]{\sz\blue $Q_1(\varDelta\wt u\,)$}
  \psfrag{P1}[lb]{\sz\red $P_1(0)$}
  \psfrag{P2}[rb]{\sz\red $P_2(2\mskip 1mu\varDelta\wt u)$}
  \psfrag{y}[rt]{\sz $y$}
  \psfrag{c}[lt]{\blue $\boldsymbol{c}$}
  \psfrag{e}[rb]{\red $\boldsymbol{e}$} % 70mm
  \includegraphics[width=50mm]{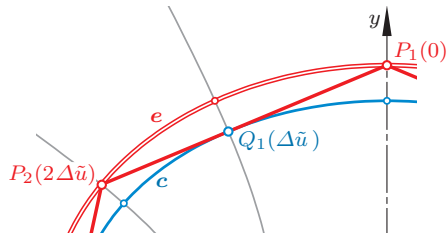}
  \caption{Dependence between the minor semiaxes $b_e$, $b_c$ and the intervall $\varDelta\wt u$.}
  \label{fig:be_bc}
\end{figure}

\begin{proof}
If the first vertex $P_1\in e$ of the billiard is chosen on the positive $y$-axis, i.e., with canonical parameter $\wt u = 0$  (see \Figref{fig:be_bc}), then the first contact point $Q_1$ has the parameter $\varDelta\wt u$, and the tangent to $c$ at $Q_1$ passes through $P_1 = (0,b_e)$.
Hence, the points $P_1$ and $Q_1$ are conjugate w.r.t.\ $c$, which means by \eqref{eq:P in tQ} that the product of the respective $y$-coordinates $b_e$ and $b_c\Cn(\varDelta\wt u)$ equals $b_c^2\,$.

\noindent
In view of the major semiaxis $a_e$ follows from \,$\mathrm{dn}^2(\varDelta\wt u) = 1 - m^2\,\mathrm{sn}^2(\varDelta\wt u)$\, and \,$\mathrm{sn}^2(\varDelta\wt u) + \mathrm{cn}^2(\varDelta\wt u) = 1$\, that \,$\Dn(\varDelta\wt u) = a_e\Cn(\varDelta\wt u)/a_c\,$.
\end{proof}

\begin{cor}\label{cor:e^1 e^2} % Cor.3
If in the ellipse $e$ with semiaxes $(a_e,b_e)$ the billiard with caustic $c$ is $N$-periodic with turning number $\tau = 1$ and $\varDelta\wt u = \frac{2 K}N\,$, then the associated Poncelet grid contains the ellipses $e^{(1)},\,e^{(2)},\dots,\,e^{(k)}$, $k = \left[\frac{N-3}2\right]$, with respective semiaxes
\[ \begin{array}{c}
    a_{e|1} = \Frac{a_c \Dn(2\mskip 1mu\varDelta\wt u)}
    {\Cn(2\mskip 1mu\varDelta\wt u)}, \ 
     b_{e|1} = \Frac{b_c}{\Cn(2\mskip 1mu\varDelta\wt u)},
    \zwi
    a_{e|2} = \Frac{a_c \Dn(3\mskip 1mu\varDelta\wt u)}
    {\Cn(3\mskip 1mu\varDelta\wt u)}, \
     b_{e|2} = \Frac{b_c}{\Cn(3\mskip 1mu\varDelta\wt u)},
    \\[3.0mm]
    \dots \ , \ a_{e|k} = \Frac{a_c \Dn(k\mskip 1mu\varDelta\wt u)}
     {\Cn(k\mskip 1mu\varDelta\wt u)}, \zwi 
     b_{e|k} = \Frac{b_c}{\Cn(k\mskip 1mu\varDelta\wt u)}\,.
  \end{array}
\]
\end{cor}  

\begin{figure}[htb] % Fig.6
  \centering
  \def\ssz{\tiny}
  \psfrag{1}[rc]{\contourlength{1.2pt}\contour{white}{\sz\blue $Q_0$}}
  \psfrag{2}[lt]{\contourlength{1.2pt}\contour{white}{\sz\blue $Q_2$}}
  \psfrag{3}[lb]{\contourlength{1.2pt}\contour{white}{\sz\green $S_1^{(1)}$}}
  \psfrag{4}[ct]{\contourlength{1.2pt}\contour{white}{\sz\blue $Q_1$}}
  \psfrag{5}[ct]{\contourlength{1.2pt}\contour{white}{\sz\green $S_0^{(1)}$}}
  \psfrag{6}[rc]{\contourlength{1.2pt}\contour{white}{\sz\green $S_2^{(1)}$}}
  \psfrag{7}[lc]{\contourlength{1.2pt}\contour{white}{\sz\red $P_1$}}
  \psfrag{8}[lb]{\contourlength{1.2pt}\contour{white}{\sz\red $P_2$}}
  \psfrag{9}[cb]{\contourlength{1.2pt}\contour{white}{\sz\red $P_2'$}}
  \psfrag{11}[lb]{\contourlength{1.2pt}\contour{white}{\sz\red $P_1'$}}
  \psfrag{14}[lt]{\contourlength{1.2pt}\contour{white}{\sz\blue $Q_1'$}}
  \psfrag{18}[lt]{\contourlength{1.2pt}\contour{white}{\sz\red $P_0'$}}
  \psfrag{19}[rt]{\contourlength{1.2pt}\contour{white}{\sz\blue $Q_0'$}}
  \psfrag{c}[ct]{\blue $\boldsymbol{c}$}
  \psfrag{e}[lb]{\contourlength{1.2pt}\contour{white}{\red $\boldsymbol{e}$}}
  \psfrag{e^1}[rb]{\green $\boldsymbol{e}^{(1)}$}
  \psfrag{1/2}[ct]{\sz\blue $\boldsymbol 0$} % \frac 12
  \psfrag{3/4}[ct]{\sz\blue $\frac 12$} % \frac 34
  \psfrag{1/1}[ct]{\sz\blue $\boldsymbol 1$} % 1
  \psfrag{5/4}[ct]{\sz\blue $\frac 32$} % 54
  \psfrag{3/2}[ct]{\sz\blue $\boldsymbol 2$} % \frac 32
  \psfrag{7/4}[ct]{\sz\blue $\frac 52$} % \frac 74
  \psfrag{2/1}[ct]{\sz\blue $\boldsymbol 3$}
  \psfrag{9/4}[ct]{\sz\blue $\frac 72$} % \frac 94
  \psfrag{5/2}[rt]{\sz\blue $u\!=\!\boldsymbol 4$} % $\frac 52$}
  \psfrag{00/1}[rt]{\contourlength{1.2pt}\contour{white}{\sz\red $v\!=\!\boldsymbol 0$}}
  \psfrag{01/2}[ct]{\contourlength{1.2pt}\contour{white}{\sz\red $\frac 12$}}
  \psfrag{03/4}[ct]{\contourlength{1.2pt}\contour{white}{\sz\red $\frac 34$}}
  \psfrag{01/1}[ct]{\contourlength{1.2pt}\contour{white}{\sz\red $\boldsymbol 1$}}
  \psfrag{05/4}[ct]{\contourlength{1.2pt}\contour{white}{\sz\red $\frac 54$}}
  \psfrag{03/2}[ct]{\contourlength{1.2pt}\contour{white}{\sz\red $\frac 32$}}
  \psfrag{07/4}[ct]{\contourlength{1.2pt}\contour{white}{\sz\red $\frac 74$}}
  \psfrag{02/1}[ct]{\contourlength{1.2pt}\contour{white}{\sz\red $\boldsymbol 2$}}
  \includegraphics[width=100mm]{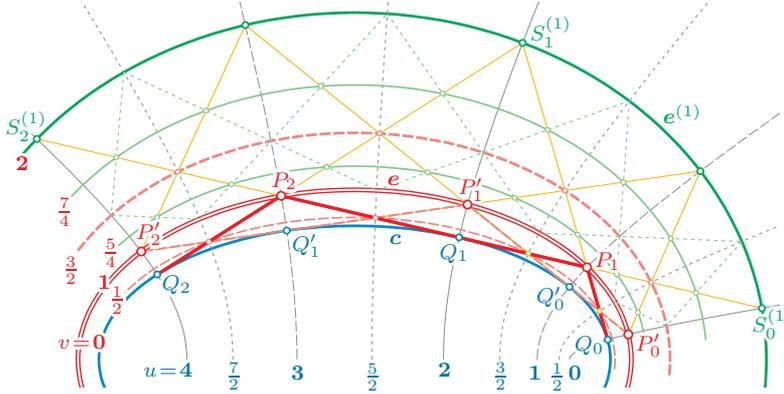} % izmest7.pas
  \caption{Canonical coordinates $u$ (blue) for the confocal hyperbolas and $v$ (red) for the confocal ellipses exterior to the caustic $c$ such that $u \pm v = \const$ represent the tangents of the caustic.}
  \label{fig:tangente2}
\end{figure}

\Corref{cor:e^1 e^2} reveals that $\varDelta\wt u$ serves as a canonical coordinate for confocal ellipses in the exterior of $c$.
If $\varDelta\wt u$ corresponds by \eqref{eq:e zu Delta_u} to the ellipse $e$ with semiaxes $(a_e,b_e)$, then $2\mskip 1mu\varDelta\wt u$ is the shift for the billiards in $e$ with caustic $c\,$. 
If these billiards have the turning number 1, then increasing the shift by $\varDelta\wt u$ means to increase the turning number of the billiard in a confocal ellipse by 1, while the caustic $c$ remains fixed  (\Figref{fig:tangente2}).
The billiard $P_1P_2\dots$ and its conjugate $P_1'P_2'\dots$ in $e$ (cf.\ \cite[Sect.~3.2]{Sta_I}) intersect each other along the ellipse with the canonical coordinate $\varDelta\wt u/2\,$.
Note that the ellipses $e^{(j)}$ and $e^{(N-2-j)}$ coincide while the corresponding $\varDelta\wt u$'s differ in their signs.
For even $N$, the points $S_i^{(N/2-1)}$ are at infinity, and the line at infinity as a limit of a confocal ellipse corresponds to $\varDelta\wt u = K$. 

The following formulas express the elliptic coordinates $(k_e,k_h)$ of the point $P = (-a_e\SN\wt u,\,b_e\CN\wt u)$ of $e$ in terms of the canonical coordinate $\wt u$ on $e$ and the shift $\varDelta\wt u$ corresponding to $e\,$.
\begin{equation}\label{eq:ke_kh}
  k_e = k_e(\varDelta\wt u) = \frac{a_c^2\,\mathrm{sn}^2\varDelta\wt u}
   {\mathrm{cn}^2\varDelta\wt u}(1 - m^2),\quad
  k_h = k_h(\wt u) = - a_c^2\,\mathrm{dn}^2\wt u\,.
\end{equation}
This follows from
\[  k_e = a_e^2 - a_c^2 
    = a_c^2\,\frac{\mathrm{dn}^2\varDelta\wt u - \mathrm{cn}^2\varDelta\wt u}
   {\mathrm{cn}^2\varDelta\wt u} 
    = a_c^2(1 - m^2)\frac{\mathrm{sn}^2\varDelta\wt u}{\mathrm{cn}^2\varDelta\wt u} 
\]
and
\[  k_h = -a_c^2\,\mathrm{cn}^2\mskip 2mu\wt u 
    - b_c^2\,\mathrm{sn}^2\mskip 2mu\wt u 
    = -a_c^2 + d^2\,\mathrm{sn}^2\mskip 2mu\wt u
    = -a_c^2 + m^2 a_c^2\,\mathrm{sn}^2\mskip 2mu\wt u\,.
\]  
Note that $k_h = k_h(\wt u)$ is a solution of \eqref{eq:diff_gl_kh}.

\medskip
In \cite{Izmestiev}, an unordered pair of coordinates $(r,s)$ is proposed for each point $P$ in the exterior of $c$, namely with $r$ and $s$ as canonical coordinates of the tangency points for the tangent lines from $P$ to $c$ (see also \cite[p.~358]{MonGeom}).
This means for $P = (-a_e\SN\wt u,\,b_e\CN\wt u)$ that 
\[ r = \wt u - \varDelta\wt u\,,\quad s = \wt u + \varDelta\wt u\,,
\]
where $\varDelta\wt u$ corresponds to $e$ according to \eqref{eq:e zu Delta_u}.

If we keep the sum $\wt u + \varDelta\wt u$ or the difference $\wt u - \varDelta\wt u$ constant, the corresponding point $P$ runs along a tangent of the caustic $c$ (compare with \cite[Prop.~8.3]{Bobenko}).
For the sake of simplicity, we replace in the summary below $\varDelta\wt u$ by $\wt v\,$.

\begin{figure}[htb] % Fig.7
  \centering
  \def\sz{\small}
  \psfrag{c}[ct]{\blue $\boldsymbol c$}
  \psfrag{e}[lc]{\red $\boldsymbol e$}
  \psfrag{e^1}[lc]{\red $\boldsymbol e^{(1)}$}
  \psfrag{e^2}[lc]{\green $\boldsymbol e^{(2)}$}
  \psfrag{e^3}[lc]{\green $\boldsymbol e^{(3)}\!=\!\boldsymbol e^{(4)}$}
  \psfrag{inf}[lc]{\sz $\infty$}
  \psfrag{4K}[ct]{\sz $4K$}
  \psfrag{K}[rc]{\sz $K$}
  \psfrag{u}[lt]{\sz $\wt u$}
  \psfrag{v}[rb]{\sz $\wt v$}
  \psfrag{Q9}[ct]{\sz\blue $Q_9$}
  \psfrag{Q1}[ct]{\sz\blue $Q_1$}
  \psfrag{Q2}[ct]{\sz\blue $Q_2$}
  \psfrag{P1}[lc]{\contourlength{1.2pt}\contour{white}{\sz\red $P_1$}}
  \psfrag{P2}[lc]{\contourlength{1.2pt}\contour{white}{\sz\red $P_2$}}
  \psfrag{P3}[lc]{\contourlength{1.2pt}\contour{white}{\sz\red $P_3$}}
  \psfrag{S1^1}[lc]{\contourlength{1.2pt}\contour{white}{\sz\dgrau $S_1^{(1)}$}}
  \psfrag{S2^1}[lc]{\contourlength{1.2pt}\contour{white}{\sz\dgrau $S_2^{(1)}$}}
  \psfrag{S1^2}[lc]{\contourlength{1.2pt}\contour{white}{\sz\green $S_1^{(2)}$}}
  \psfrag{S2^2}[lc]{\contourlength{1.2pt}\contour{white}{\sz\green $S_2^{(2)}$}}
  \psfrag{S1^3}[lc]{\contourlength{1.2pt}\contour{white}{\sz\green $S_1^{(3)}$}}
  \psfrag{S2^3}[lc]{\contourlength{1.2pt}\contour{white}{\sz\green $S_2^{(3)}$}}
  \includegraphics[width=110mm]{\pfad 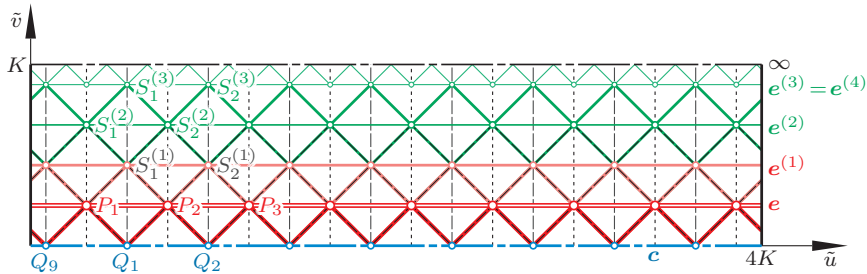} % CAD_2D: grid_b
  \caption{The injective mapping $\Vkt Y$ sends the square grid of points $Q_i, P_i$ and $S_i^{(j)}$, $i=1,\dots,9\,$, $j=1,\dots,3\,$, to the vertices and the diagonals to the confocal conics of the Poncelet grid depicted in \Figref{fig:Poncelet_grid2}.}
\label{fig:map}
\end{figure}

\begin{thm}\label{thm:PonceletGrid} % Thm.4
Referring to the notation in \Thmref{eq:action}, the injective mapping
\[ \begin{array}{rl}
    \Vkt Y\!:\ &U\times V \to \RR^2, \quad 
    (\wt u,\,\wt v) \,\mapsto\, \left(-a_c\Frac{\SN\wt u\,\DN\wt v}{\CN\wt v}, \ 
     b_c\Frac{\CN\wt u}{\CN\wt v}\right)
    \\[3.0mm]
    &\mathrm{for}\zwi U:= \{\wt u\,|\,0\le \wt u < 4K\},\
       V:= \{\wt v\,|\,0\le \wt u < K\} 
   \end{array} 
\]
parametrizes the exterior of the caustic $c$ with semiaxes $(a_c,b_c)$ in such a way, that the lines $\wt u = \const$ are branches of confocal hyperbolas; $\wt v = \const$ are confocal ellipses and $\wt u \pm \wt v = \const$ tangents of $c\,$.\end{thm}

The domain of the mapping $\Vkt Y$ (\Figref{fig:map}) % ref neu in v2
can be extended to $\RR^2$ and satisfies
\[ \Vkt Y((\wt u + 4K),\,\wt v) = \Vkt Y(\wt u,\,(\wt v + 2K))
    = \Vkt Y(\wt u,\,-\wt v) = \Vkt Y(\wt u,\,\wt v) 
\]
and therefore $\Vkt Y(\wt u,\,(K + \wt v)) = \Vkt Y(\wt u,\,(K - \wt v))$
(\Figref{fig:map}).
The Liegroup $\Gamma$ mentioned in \Thmref{thm:vel} is the $\Vkt Y$-transform of the group of translations along the $\wt u$-axis.

% -----------------------------------------------------------------------------
\section{More about invariants of periodic billiards}
%       ---------------------------------------------

In this section we study how the infinitesimal motion induced by the vector field in \eqref{eq:inf_motion1} affects distances and angles of periodic billiards.
As before, the dot means differentiation by the canonical parameter $u$, and we call the billiard motion {\em canonical} when it is parametrized by a canonical parameter like $u\,$.

\begin{lem}\label{lem:liri} % Lem.5
Let $P_1P_2\dots$ be a billiard in the ellipse $e$ with $Q_1,Q_2,\dots$ as contact points with its caustic, the ellipse $c$.
If $t_i$ is the parameter of $P_i$ and $t_i'$ that of $Q_i$, then 
\[ r_i = \ol{Q_{i-1}P_i} = \Frac{\Vert\Vkt t_c(t_{i-1}')\Vert\,\Vert\Vkt t_c(t_i)\Vert \sqrt{k_e}}
          {a_c b_c}, \quad 
    l_i = \ol{P_iQ_i} = \Frac{\Vert\Vkt t_c(t_i)\Vert\,\Vert\Vkt t_c(t_i')\Vert 
          \sqrt{k_e}}{a_c b_c}.
\]
The canonical motion of the billiard induces for the side $P_iP_{i+1}$ the instant angular velocity % about $Q_i$ as 
\[  \omega_i = \frac{a_c b_c}{\Vert\Vkt t_c(t_i')\Vert}
     = \,\frac{a_c b_c}{\sqrt{-k_h(t_i')}}\,. %=\ol{O\,[P_{i-1},P_i]}. 
\]
\end{lem}

\begin{figure}[hbt] % Fig.8
  \centering 
%   \psfrag{F1}[rt]{\sz $F_2$}
%   \psfrag{F2}[lt]{\sz $F_1$}
  \psfrag{P}[cb]{\contourlength{1.2pt}\contour{white}{\sz\red $P_2$}}
  \psfrag{P'}[lb]{\contourlength{1.2pt}\contour{white}{\sz\blue $Q_1'$}}
  \psfrag{T1}[lb]{\sz\blue $Q_1$}
  \psfrag{T2}[rb]{\contourlength{1.4pt}\contour{white}{\sz\blue $Q_2$}}
  \psfrag{P1'}[lb]{\contourlength{1.2pt}\contour{white}{\sz\red $P_1'$}}
  \psfrag{P2'}[rb]{\sz\red $P_2'$}
  \psfrag{T1*}[rc]{\sz\blue $Q_1^*$}
  \psfrag{T2*}[lt]{\contourlength{1.4pt}\contour{white}{\sz\blue $Q_2^*$}}
  \psfrag{la}[cb]{\sz\blue $r_2$}
  \psfrag{e}[rt]{\red $\boldsymbol{e}$}
  \psfrag{e0}[rt]{\blue $\boldsymbol{c}$}
  \psfrag{vn1}[rt]{\sz $\Vkt v_{n_1}$}
  \psfrag{vt1}[rb]{\sz $\Vkt v_{t_1}$}
  \psfrag{vn2}[rb]{\sz $\Vkt v_{n_2}$}
  \psfrag{vt2}[rt]{\sz $\Vkt v_{t_2}$}
  \psfrag{v}[rb]{\sz $\Vkt v_2$}
  \psfrag{v'}[rc]{\contourlength{1.4pt}\contour{white}{\sz $\Vkt v_1'$}}
  \psfrag{om1}[ct]{\sz\red $\omega_1$}
  \psfrag{om2}[lt]{\sz\red $\omega_2$}
  \psfrag{r2}[lb]{\sz\red $r_2$}
  \psfrag{l2}[lt]{\sz\red $l_2$}
  \psfrag{rho1}[lt]{\contournumber{32}\contourlength{1.4pt}\contour{white}{\sz\blue $\rho_c(t_1')$}} 
  \includegraphics[width=80mm]{\pfad 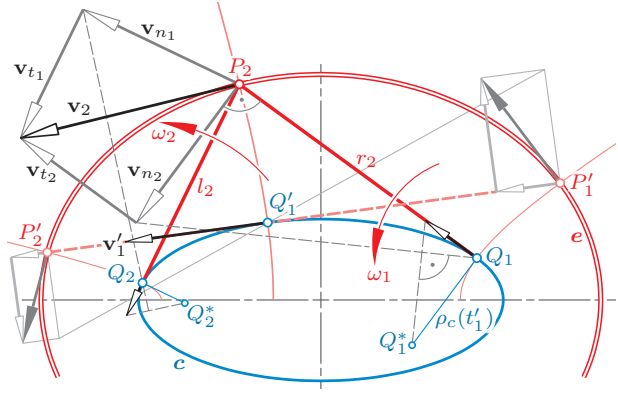} % CAD_2D: graves_c
  \caption{Velocities $\Vkt v_2$ of the vertex $P_2$, $\Vkt v_1'$ of the affine image $Q_1'$ with $\Vert\Vkt v_1'\Vert = \Vert\Vkt v_{n_1}\Vert = \Vert\Vkt v_{n_2}\Vert$, of the contact points $Q_i$ (with $Q_i^*$ as respective centers of curvature), and of vertices $P_i'$ of the conjugate billiard, $i=1,2\,$.}
  \label{fig:Graves2}
\end{figure}

\begin{proof}
Referring to \Figref{fig:Graves2}, if the tangent $[Q_1,P_2]$ rolls on $c$, then the vertex $P_2$ receives the velocity vector $\Vkt v_{n_2}$ satisfying \eqref{eq:vt, v global}, while the point of contact $Q_1$ moves with the velocity $v_c(t_1')$ along $c\,$. 
We can express this velocity in terms of the radius of curvature $\rho_c(t_1')$ of $c$ as 
\[ v_c(t_1') = \omega_1\,\rho_c(t_1'),
\]
where $\rho_c(t) =\Vert\Vkt t_e(t)\Vert^3/a_c b_c$ by \cite[p.~79]{Conics}.
On the other hand, from $\Vkt v_c = \dot t\,\Vkt t_c$ and \eqref{eq:dot t} follows $v_c(t_1') = \Vert\Vkt t_c(t_1')\Vert^2$.  
This yields
\[  \omega_1 = \frac{v_c(t_1')}{\rho_c(t_1')} 
    = \frac{\Vert\Vkt t_c(t_1')\Vert^2\,a_c b_c}{\Vert\Vkt t_c(t_1')\Vert^3}
    = \frac{a_c b_c}{\Vert\Vkt t_c(t_1')\Vert}
    = \frac{a_c b_c}{\sqrt{-k_h(t_1')}}
\]
by \eqref{eq:k_h}.
Thus, we obtain for the velocity $v_{n|2}$ of $P_2$ by \eqref{eq:vt, v global}
\[  r_2\,\omega_1 = v_{n|2} = \Vert\Vkt t_c(t_2)\Vert \sqrt{k_e}, 
    \zwi\mbox{hence}\zwi
    r_2 = \frac{\Vert\Vkt t_c(t_2)\Vert \sqrt{k_e}}{a_c b_c}\,\Vert\Vkt t_c(t_1')\Vert.
\]
Similarly follows from $l_2\,\omega_2 = v_{n|2}$ the stated expression for the distance $l_2\,$.
Note that $t_1, t_1', t_2, t_2', t_3,\dots$ is the sequence of consecutive parameters of the points $P_1,Q_1,P_2,Q_2,P_3,\dots$
The formulas for $r_i$ and $l_i$ reveal again that the same distances appear at the conjugate billiard.
\end{proof}

\medskip
The angular velocity of the tangent to $e$ at $P_2$ equals the arithmetic mean $(\omega_1+\omega_2)/2$ (\Figref{fig:Graves2}).
On the other hand, it is defined by the radius of curvature $\rho_e$ of $e$ at $P_2$ and the velocity $v_2$ by \eqref{eq:vt, v global}, since
\[  v_2 = \rho_e(t_2) \frac{\omega_1+\omega_2}2\,. 
\]
This means by \Lemref{lem:liri}
\[  \Vert\Vkt t_c(t_2)\Vert\,\Vert\Vkt t_e(t_2)\Vert
    = \Frac{\Vert\Vkt t_e(t_2)\Vert^3}{a_e b_e}\,\Frac{a_c b_c}2
    \left(\Frac 1{\Vert\Vkt t_c(t_1')\Vert} + \Frac 1{\Vert\Vkt t_c(t_2')\Vert}\right)
\]
and results by \eqref{eq:Winkel/2} in
\begin{equation}\label{eq:t1't2t2'}
  \frac 1{\Vert\Vkt t_c(t_1')\Vert} + \frac 1{\Vert\Vkt t_c(t_2')\Vert}
  = \frac{2 a_e b_e}{a_c b_c}\,\frac{\Vert\Vkt t_c(t_2)\Vert}
    {\Vert\Vkt t_e(t_2)\Vert^2}
  = \frac{a_e b_e}{a_c b_c\sqrt{k_e}}\,\sin\theta_2\,.  
\end{equation}

\begin{thm}\label{thm:altern_sum} % Thm.6
 The exterior angles $\theta_i$ of an $N$-periodic billiard in an ellipse and with an ellipse as caustic satisfy for $N\equiv 0\pmod 2$ % even $N$ 
\[  \sum_{i=1}^N (-1)^i \sin\theta_i = 0
    \zwi\mbox{and for}\zwi N\equiv 0\mskip -15mu\pmod 4 \zwi
    \sum_{i=1}^{N/2} (-1)^i \sin\theta_i = 0\,.
\]
\end{thm}

\begin{proof}
By virtue of \cite[Corollary~4.2]{Sta_I}, periodic billiards with even $N = 2n$ are centrally symmetric, which implies $\theta_i = \theta_{i+n}\,$.
For $N\equiv 2\pmod 4$ the sum from $1$ to $N$ must vanish since $(-1)^i = -(-1)^{(i+n)}$.
\\
In the remaining case $N\equiv 0\pmod 4$ follows from \eqref{eq:t1't2t2'}  
\[ \sin\theta_i = \frac{a_c b_c\sqrt{k_e}}{a_e b_e}\left(
  \frac 1{\Vert\Vkt t_c(t_{i-1}')\Vert} + \frac 1{\Vert\Vkt t_c(t_i')\Vert}\right)
\]
and further
\[ \begin{array}{l}
   \displaystyle\sum_{i=1}^{N/2} \,\sin\theta_i = \Frac{a_c b_c \sqrt{k_e}}{a_e b_e}
   \,\cdot \\
   \phantom{m}\left(\Frac 1{\Vert\Vkt t_c(t_N')\Vert} 
   + \Frac 1{\Vert\Vkt t_c(t_1')\Vert} 
   - \Frac 1{\Vert\Vkt t_c(t_1')\Vert} 
   - \Frac 1{\Vert\Vkt t_c(t_2')\Vert} 
   + - \dots - \Frac 1{\Vert\Vkt t_c(t_n')\Vert}\right).
  \end{array} 
\]
This sum vanishes, since $\Vkt t_c(t_n') = -\Vkt t_c(t_N')\,$, due to the odd turning number $\tau$ because of $\mbox{gcd}(N,\tau)=1\,$.
\end{proof}   

\medskip
At the same token, from 
\begin{equation}\label{eq:angle_dot}
  \dot{\theta}_i = \omega_i - \omega_{i-1} 
  = a_c b_c \left( \frac 1{\Vert\Vkt t_c(t_i')\Vert} 
    - \frac 1{\Vert\Vkt t_c(t_{i-1}')\Vert}\right)
\end{equation}
and \eqref{eq:t1't2t2'} follows
\[  \frac{\mathrm d}{\mathrm d u}\cos\theta_i = -\dot\theta_i \sin\theta_i
    = \frac{a_c^2 b_c^2 \sqrt{k_e}}{a_e b_e}\left(
    \frac 1{\Vert\Vkt t_c(t_{i-1}')\Vert^2} 
    - \frac 1{\Vert\Vkt t_c(t_i')\Vert^2}\right).
\]
This shows that $\frac{\mathrm d}{\mathrm d u}\left(\sum_1^N \cos\theta_i\right)$ vanishes and, therefore, $\sum_1^N \cos\theta_i$ is invariant against billiard motions, which was first proved in \cite{Ako-Tab}.

For the sake of completeness, we also focus on the variation of the side lengths under the canonical billiard motion.
We obtain
\begin{equation}\label{eq:side}
 \begin{array}{rcl} 
  \Frac{\mathrm d}{\mathrm d u}\,\ol{P_iP_{i+1}} &= &v_{t|i+1} - v_{t|i}
  \\[1.5mm]
   &= &\Vert\Vkt t_c(t_{i+1})\Vert^2 - \Vert\Vkt t_c(t_i)\Vert^2
       = d^2(\sin^2 t_{i+1} - \sin^2 t_i).
  \end{array}                         
\end{equation}
The vanishing sum over all $i$ confirms again the constant perimeter.
We like to recall that already in \cite{Bialy-Tab} some proofs for invariants were based on differentiation.

\medskip
Finally we concentrate on the effects showing up when the vertex $P_i$ traverses a quarter of the full period along the billiard.
 
\begin{lem}\label{lem:Viertel}  % Lem.7
As before, let $t_1,t_1',t_2,t_2',\dots, t_N'$ be the sequence of parameters of an $N$-periodic billiard in an ellipse $e$ with an ellipse $c$ as caustic.
Then holds for $N\equiv 0\pmod 2$: % even $N$:
\[ \left. \begin{array}{rl}
    \mbox{if}\zwi N = 4n:\zwi 
     &\Vert\Vkt t_c(t_i)\Vert\,\Vert\Vkt t_c(t_{i+n})\Vert 
%      = \Vert\Vkt t_c(t_i')\Vert\,\Vert\Vkt t_c(t_{i+n}')\Vert 
      = \sqrt{k_h(t_i)\, k_h(t_{i+n})} 
    \\[1.0mm]
    \mbox{if}\zwi N = 4n\!+\!2:\zwi 
     &\Vert\Vkt t_c(t_i)\Vert\,\Vert\Vkt t_c(t_{i+n}')\Vert 
%      = \Vert\Vkt t_c(t_i')\Vert\,\Vert\Vkt t_c(t_{i+n+1})\Vert 
      = \sqrt{k_h(t_i)\, k_h(t_{i+n}')} 
  \end{array}\right\} = a_c\,b_c\,,
\]
and the same after the parameter shift $t_i\mapsto t_i'$ and $t_i'\mapsto t_{i+1}\,$.
\end{lem}

\begin{proof}
Based on the canonical parametrization by $\wt u$, a quarter of the period $4K$ corresponds to a shift by $K$.
In the case $N = 4n$ this shift effects $t_i\mapsto t_{i+n}$ and $t_i'\mapsto t_{i+n}'$.
If $N = 4n+2$, then $t_i\mapsto t_{i+n}'$ and $t_i'\mapsto t_{i+n+1}$. 
\\
According to \eqref{eq:ke_kh} holds $k_h = -a_c^2\,\mathrm{dn}^2\wt u$ and by \eqref{eq:k_h} $\Vert \Vkt t_c(t)\Vert = \sqrt{-k_h(t)} = a_c\DN\wt u\,$.
The identity 
\[  \Dn(\wt u + K) = \frac{\sqrt{1-m^2}}{\Dn(\wt u)}
\]
implies
\[  \Dn(\wt u)\cdot\Dn(\wt u + K) = \frac{b_c}{a_c}\,,\zwi\mbox{hence}\zwi
   \sqrt{k_h(\wt u) \cdot k_h(\wt u + K)} = a_c\,b_c\,.
\]
This confirms the claim.
\end{proof}

\begin{thm}\label{thm:Viertel}  % Thm.8
If an $N$-periodic billiard in an ellipse $e$ with an ellipse as caustic is given with $N\equiv 0\pmod 2$, then the distances $r_i = \ol{Q_{i-1}P_i}$ and $l_i = \ol{P_iQ_i}$ satisfy
\[ \left. \begin{array}{rl}
    \mbox{for}\zwi N = 4n:\zwi        &r_i\cdot r_{i+n} = l_i \cdot l_{i+n} 
    \\[1.0mm]
    \mbox{for}\zwi N = 4n\!+\!2:\zwi  &r_i\cdot l_{i+n} = l_i \cdot r_{i+n+1}
  \end{array} \right\} = k_e\,.
\]
\end{thm}

\begin{proof}
From the expressions for $r_i$ and $l_i$ in \Lemref{lem:liri} follows, by virtue of \Lemref{lem:Viertel}, for $N = 4n$
\[  r_i\cdot r_{i+n} = \frac{k_e}{a_c^2 b_c^2}\,\Vert\Vkt t_c(t_{i-1}')\Vert\,
    \Vert\Vkt t_c(t_i)\Vert\,\Vert\Vkt t_c(t_{i+n-1}')\Vert\,
    \Vert\Vkt t_c(t_{i+n})\Vert = k_e
\]
and the same result for $l_i\cdot l_{i+n}\,$.
In the case $N = 4n + 2$ we obtain similarly
\[  r_i\cdot l_{i+n} = l_i\cdot r_{i+n+1}  = k_e\,,
\]
as stated.
\end{proof}

% neu in v2
The following corollary is an immediate consequence of \Thmref{thm:Viertel}.

\begin{cor}\label{cor:verh_si}
Let $s_i = \ol{P_iP_{i+1}} = l_i + r_{i+1}$ for $i=1,\dots,N$ be the side lengths of the $N$-periodic billiard with even $N$ and $s_i' = \ol{P_i'P_{i+1}'} = r_{i+1} + l_{i+1}$ that of the conjugate billiard.
Then, 
\[ \begin{array}{rl} 
    \mbox{for}\zwi N = 4n: \quad &\Frac{s_{i+n}}{s_i} 
%      = \Frac{l_{i+n} + r_{i+n+1}}{l_i + r_{i+1}}
     = \Frac{l_{i+n}}{r_{i+1}} = \Frac{r_{i+n+1}}{l_i}\,, 
    \\[2.5mm]
    \mbox{for}\zwi N = 4n\!+\!2: \quad &\Frac{s_{i+n}}{s_{i-1}'} 
%      = \Frac{l_{i+n} + r_{i+n+1}}{r_i + l_i}
     = \Frac{l_{i+n}}{l_i} = \Frac{r_{i+n+1}}{r_i}\,.
  \end{array}
\]
\end{cor}

\bigskip
Finally we prove the invariance of k117 in \cite[Table~2]{80}.

\begin{thm}\label{thm:prod_li} % Thm.9
Referring to the notation in \Lemref{lem:liri}, for even $N$ the products
\[  r_1 r_2 \dots r_N = l_1 l_2\dots l_N = k_e^{N/2}  
\]
are invariant against billiard motions.
For $N\equiv 0\pmod 4$ this is already true for the products
\[  r_1 r_2 \dots r_{N/2} = l_1 l_2\dots l_{N/2} = k_e^{N/4}.  
\]
\end{thm}

\begin{proof}
For $N\equiv 0\pmod 4$ the statements are a direct consequence of \Thmref{thm:Viertel} and the central symmetry of the billiard which exchanges $r_i$ with $r_{i+N/2}$ and $l_i$ with $l_{i+N/2}\,$.
\\
In the remaining case $N = 2n+2$ we note that by \eqref{eq:r1..rN} $R(u):= r_1 r_2\dots r_N = l_1 l_2\dots l_N$.
Hence, by virtue of \Thmref{thm:Viertel}, 
\[  R^2(u) = \prod_{i=1}^N \ (r_i\,l_{i+n}) = k_e^N,
\]
which yields the stated result.
\end{proof}

% ------------------------------------------------------------------------------

% \section{Conclusion}

% The conclusion goes here.

% ------------------------------------------------------------------------------
% \begin{acknowledgements}
% If you'd like to thank anyone, place your comments here
% and remove the percent signs.
% \end{acknowledgements}


\end{document}